\documentclass[12pt]{article}%
\usepackage[latin9]{inputenc}
\usepackage{amsmath}
\usepackage{amssymb}
\usepackage{amsfonts}
\usepackage{graphicx}
\usepackage{color}%
\providecommand{\U}[1]{\protect\rule{.1in}{.1in}}
\newtheorem{theorem}{Theorem}

\newtheorem{corollary}[theorem]{Corollary}

\newtheorem{definition}[theorem]{Definition}
\newtheorem{example}[theorem]{Example}
\newtheorem{fact}[theorem]{Fact}

\newenvironment{proof}[1][Proof]{\noindent\textbf{#1.} }{\ \rule{0.5em}{0.5em}}
\makeatother
\begin{document}

\title{\textbf{On Componental Operators }\\\textbf{in Hilbert Space}}
\author{Andrzej Cegielski$^{1}$ and Yair Censor$^{2}\bigskip$\\$^{1}$Institute of Mathematics\\University of Zielona Góra\\Zielona Góra, Poland\\(a.cegielski@wmie.uz.zgora.pl)$\bigskip$\\$^{2}$Department of Mathematics\\University of Haifa\\Mt.\ Carmel, Haifa 3498838, Israel\\(yair@math.haifa.ac.il) }
\date{July 15th, 2021. \\
Revised: November 2nd, 2021 and November 20, 2021.}
\maketitle

\begin{abstract}
We consider a Hilbert space that is a product of a finite number of Hilbert
spaces and operators that are represented by ``componental operators''\ acting
on the Hilbert spaces that form the product space. We attribute operatorial
properties to the componental operators rather than to the full operators. The
operatorial properties that we discuss include nonexpansivity, firm
nonexpansivity, relaxed firm nonexpansivity, averagedness, being a cutter,
quasi-nonexpansivity, strong quasi-nonexpansivity, strict quasi-nonexpansivity
and contraction.

Some relationships between operators whose componental operators have such
properties and operators that have these properties on the product space are
studied. This enables also to define componental fixed point sets and to study
their properties. For componental contractions we offer a variant of the
Banach fixed point theorem.

Our motivation comes from the desire to extend a fully-simultaneous method
that takes into account sparsity of the linear system in order to accelerate
convergence {[}Censor et al., On diagonally relaxed orthogonal projection
methods, \textit{SIAM J. Sci. Comput.} \textbf{30} (2008), 473--504{]}. This
was originally applicable to the linear case only and gives rise to an
iterative process that uses different componental operators during iterations.

\end{abstract}

\section{Introduction\label{sec:intro}}

Let $\mathcal{H}_{j}$, $j\in J:=\{1,2,...,n\},$ be real Hilbert spaces and
$\mathcal{H}:\mathcal{=H}_{1}\times\mathcal{H}_{2}\times\cdots\times
\mathcal{H}_{n}$ for $n\geqslant2$. Consider an operator $T:\mathcal{H}%
\rightarrow\mathcal{H}$ that maps the Hilbert space $\mathcal{H}$ into itself
and is represented by
\begin{equation}
T(x)=((T(x))_{1},(T(x))_{2},...,(T(x))_{n})\text{,}%
\end{equation}
where, for all $j\in J,$ $(T(\cdot))_{j}:\mathcal{H}\rightarrow$
$\mathcal{H}_{j}$ denotes \textquotedblleft the $j$-th componental operator of
$T(\cdot)$\textquotedblright.

In this paper we attribute operatorial properties to the componental operators
of an operator $T$. To explain what we mean, let us take, by way of example,
the well-known nonexpansivness property. We define that the operator $T$ is
\textquotedblleft$j$\textit{-nonexpansive} ($j$-NE)\textquotedblright\ if for
all $x,y\in\mathcal{H}$
\begin{equation}
\Vert(T(x))_{j}-(T(y))_{j}\Vert_{j}\leq\Vert x_{j}-y_{j}\Vert_{j}\text{,}%
\end{equation}
for some $j\in J,$ where $\Vert\cdot\Vert_{j}$ is the norm in $\mathcal{H}%
_{j}.$ If $T$ is $j$-NE for all $j\in J$ then we say that $T$ is
\textquotedblleft component-wise NE (CW-NE)\textquotedblright.

Obviously, such definitions imply that if $T$ is CW-NE then it is NE but not
vice versa, making the set of CW-NE operators a proper subset of the NE
operators. In this sense, a CW-NE operator is ``stronger''\ than an NE
operator which is not necessarily CW-NE.

We look at various operatorial properties of the componental operators of an
operator $T$ and investigate their relationships with such properties of the
operator $T$ itself. To do so, we define the notions of $j$%
\textit{-nonexpansive} ($j$-NE), $j$-\textit{firmly nonexpansive} ($j$-FNE),
$j$-\textit{relaxed firmly nonexpansive} ($j$-RFNE), $j$-\textit{averaged}
($j$-AV), $j$\textit{-cutter, }$j$\textit{-quasi-nonexpansive} ($j$-QNE),
$(\rho_{j},j)$\textit{-strongly quasi-nonexpansive} ($(\rho_{j},j)$-SQNE),
$j$\textit{-strictly quasi-nonexpansive} ($j$-sQNE) and $(\alpha_{j}%
,j)$\textit{-contraction} ($(\alpha_{j},j)$-CONT) and their associated
\textquotedblleft CW-X\textquotedblright\ properties, where \textquotedblleft
X\textquotedblright\ stands for any of the, above mentioned, properties: NE,
FNE, RFNE, AV, cutter, QNE, SQNE, sQNE or CONT. To say that $T$ is CW-X means
that it is \textquotedblleft$j$-X\textquotedblright\ for all $j\in J.$

For $j\in J$ we also define the \textquotedblleft$j$-th fixed point set of
$T$\textquotedblright\ by
\begin{equation}
\operatorname*{Fix}{}^{j}T:=\{z\in\mathcal{H}\mid(T(z))_{j}=z_{j}\}
\end{equation}
and study its properties.

Our motivation comes from a desire to study an iterative process that uses, as
the iterations proceed, different componental operators of the operator $T.$
We devote below a special section to describe this motivating topic.

Working with ``blocks''\ of variables is a common and wide-ranging subject in
the study of iterative processes. An early approach to analyzing iterative
processes over Cartesian product sets and ``component solution methods''\ can
be found in the book of Bertsekas and Tsitsiklis \cite[Subsection
3.1.2]{bertsekas-book}. A product space formulation, which became classical by
now, for feasibility-seeking and optimization problems, is the work of Pierra
\cite{Pierra1984}. A kind of stochastic component solution methods were
studied in \cite{CP18}.

Here we set forth a general framework that fits a large variety of operatorial
properties. It remains to be discovered whether the subsets of operators of
the form ``CW-X'', for any of the above mentioned properties ``X'', can
generate stronger results in fixed point theory due to the fact that a CW-X
operator is ``stronger''\ than its associated operator with property X which
is not CW-X.

The paper is structured as follows. In Subsections \ref{subsec:NEs} and
\ref{sec:CW-QNE} we develop our framework of componental properties of
operators. In Subsection \ref{subsec:contract} we define $\alpha_{j}%
$\textit{-contractions }and formulate\textit{ }the Banach fixed point theorem
for them. In Section \ref{sec:regularity} we define and study the componental
regularity property, and in Section \ref{sec:nonlin-drop} we present and
analyze our motivating case of extending the fully-simultaneous
Diagonally-Relaxed Orthogonal Projections (DROP) method of \cite{CEHN08},
originally applicable to the linear case only.

\section{Componental operators\label{sec:CW-NE}}

Let $\mathcal{H}_{j}$ be a real Hilbert space with inner product $\langle
\cdot,\cdot\rangle_{j}$ and induced norm $\Vert\cdot\Vert_{j}$, $j\in
J:=\{1,2,...,n\}$. Define the product Hilbert space $\mathcal{H}%
:=\mathcal{H}_{1}\times\cdots\times\mathcal{H}_{n}$ with inner product
$\langle\cdot,\cdot\rangle:\mathcal{H\times H}\rightarrow%
\mathbb{R}
$ defined by $\langle x,y\rangle$ $:=\sum_{j=1}^{n}\langle x_{j},y_{j}%
\rangle_{j}$ and the induced norm $\Vert\cdot\Vert:\mathcal{H}\rightarrow%
\mathbb{R}
$ defined by $\Vert x\Vert:=\sqrt{\sum_{j=1}^{n}\Vert x_{j}\Vert_{j}^{2}}$,
where $x=(x_{1},x_{2},...,x_{n})\in\mathcal{H}$, $y=(y_{1},y_{2},...,y_{n}%
)\in\mathcal{H}$ with $x_{j},y_{j}\in\mathcal{H}_{j}$, $j\in J$. Let
$T:\mathcal{H}\rightarrow\mathcal{H}$, i.e.,
\begin{equation}
T(x)=((T(x))_{1},(T(x))_{2},...,(T(x))_{n})\text{,}\label{e-T}%
\end{equation}
where $(T(x))_{j}\in\mathcal{H}_{j}$ denotes \textquotedblleft the $j$-th
componental operator of $T(x)$\textquotedblright, $j\in J$. Denote by
$T_{\lambda}:=\operatorname*{Id}+\lambda(T-\operatorname*{Id})$ the $\lambda
$-relaxation of $T$, where $\lambda\geq0$ and $\operatorname*{Id}$ denotes the identity.

\subsection{Componental nonexpansive, firmly nonexpansive, relaxed firmly
nonexpansive and averaged operators\label{subsec:NEs}}

Bearing in mind the well-known definitions of operators that are
\textit{nonexpansive }(NE)\textit{, firmly nonexpansive} (FNE),
\textit{relaxed firmly nonexpansive} (RFNE), or \textit{averaged} (AV), see,
e.g., \cite{cegielski-book} or \cite{bau-com-book}, we introduce the following
new definitions. We refer to any of these types of operators by the general
name of ``\textit{componental operators}''.

\begin{definition}
\label{d1}%
\rm\
Let $j\in J$. We say that an operator $T:\mathcal{H}\rightarrow\mathcal{H}$ is:

\begin{enumerate}
\item[(i)] $j$\textit{-nonexpansive} ($j$-NE) if for all $x,y\in\mathcal{H}$
\begin{equation}
\Vert(T(x))_{j}-(T(y))_{j}\Vert_{j}\leq\Vert x_{j}-y_{j}\Vert_{j}\text{;}
\label{e-j-NE}%
\end{equation}

\item[(ii)] $j$-\textit{firmly nonexpansive} ($j$-FNE) if for all
$x,y\in\mathcal{H}$
\begin{equation}
\langle(T(x))_{j}-(T(y))_{j},x_{j}-y_{j}\rangle_{j}\geq\Vert(T(x))_{j}%
-(T(y))_{j}\Vert_{j}^{2}\text{;}%
\end{equation}

\item[(iii)] $j$-\textit{relaxed firmly nonexpansive} ($j$-RFNE), if there
exist a constant $\lambda\in\lbrack0,2]$ and a $j$-FNE operator $U:\mathcal{H}%
\rightarrow\mathcal{H}$ such that $(T(\cdot))_{j}$ is a $\lambda$-relaxation
of $(U(\cdot))_{j}$;

\item[(iv)] $j$-\textit{averaged} ($j$-AV), if there exist a constant
$\alpha\in(0,1)$ and a $j$-NE operator $U$ such that
\begin{equation}
(T(x))_{j}=(1-\alpha)x_{j}+\alpha(U(x))_{j};
\end{equation}

\item[(v)] \textit{component-wise NE} (CW-FNE, CW-RFNE, CW-AV) if $T$ is
$j$-NE ($j$-FNE, $j$-RFNE, $j$-AV) for all $j\in J$.
\end{enumerate}

Alternatively, if we want to emphasize the constants $\lambda$ in (iii) and
$\alpha$ in (iv) explicitly, then we say that $T$ is $(\lambda,j)$-relaxed
firmly nonexpansive ($(\lambda,j)$-RFNE) in (iii) and $(\alpha,j)$-averaged
($(\alpha,j)$-AV) in (iv).
\end{definition}

Clearly, a CW-NE (CW-FNE) operator is NE (FNE). Note, however, that the
converse is not true.

\begin{example}%
\rm\
The operator $T:%
\mathbb{R}%
^{2}\rightarrow%
\mathbb{R}%
^{2}$ defined by $T(x_{1},x_{2})=(x_{2},x_{1})$ is NE but not CW-NE.
\end{example}

\begin{fact}
\label{f1}Let $j\in J$ and let $T$ be $j$-NE. Let $x,y\in\mathcal{H}$ be such
that $x_{j}=y_{j}$. Then
\begin{equation}
(T(x))_{j}=(T(y))_{j}\text{.} \label{e-Txj}%
\end{equation}
Thus, one can define an operator $T^{j}:\mathcal{H}_{j}\rightarrow
\mathcal{H}_{j}$ by
\begin{equation}
T^{j}(x_{j}):=(T(x))_{j}. \label{e-Tjx}%
\end{equation}

\end{fact}

\begin{proof}
Equality (\ref{e-Txj}) follows directly from Definition \ref{d1}.
\end{proof}

\bigskip{}

Fact \ref{f1} has a consequence that the componental operators of a CW-NE
operator $T$ are NE operators.

\begin{fact}
\label{f2}If an operator $T:\mathcal{H}\rightarrow\mathcal{H}$ is CW-NE, then
for any $x\in\mathcal{H}$ it holds
\begin{equation}
T(x)=(T^{1}(x_{1}),T^{2}(x_{2}),...,T^{n}(x_{n})) \label{e-decomp}%
\end{equation}
and $T^{j},j\in J$, are NE. Conversely, if the operators $T^{j}:\mathcal{H}%
_{j}\rightarrow\mathcal{H}_{j}$, $j\in J$, are NE then the operator $T$
defined by (\ref{e-decomp}) is CW-NE.
\end{fact}

\begin{proof}
In view of (\ref{e-T}), Equality (\ref{e-decomp}) follows from Fact \ref{f1}.
\end{proof}

\bigskip{}

The following corollary shows that we can set a common constant $\alpha
\in(0,1)$ and a common CW-NE operator $U$ in the definition of a CW-AV
operator $T$.

\begin{corollary}
An operator $T:\mathcal{H}\rightarrow\mathcal{H}$ is CW-AV if and only if
there is a constant $\alpha\in(0,1)$ and a CW-NE operator $U$ such that
$T=V_{\alpha}:=(1-\alpha)\operatorname*{Id}+\alpha U$.
\end{corollary}

\begin{proof}
The ``if''\ part is obvious. Suppose that $T$ is CW-AV. Then for any $j\in J$
there are $\alpha_{j}\in(0,1)$ and CW-NE operators $U_{j}:\mathcal{H}%
\rightarrow\mathcal{H}$ such that for all $x\in\mathcal{H}$ it holds
\begin{equation}
(T(x))_{j}=(1-\alpha_{j})x_{j}+\alpha_{j}(U_{j}(x))_{j}.
\end{equation}
By Fact \ref{f2}, $U_{j}(x)=(U_{j}^{1}(x_{1}),U_{j}^{2}(x_{2}),...,U_{j}%
^{n}(x_{n}))$, $j\in J$, $x\in\mathcal{H}$, and all operators $U_{j}%
^{i}:\mathcal{H}_{i}\rightarrow\mathcal{H}_{i}$, $i,j\in J$, defined by
$U_{j}^{i}(x_{i}):=(U_{j}(x))_{i}$ are NE. Let $\alpha:=\max_{j\in J}%
\alpha_{j}$ and $\beta_{j}:=\alpha_{j}/\alpha\in(0,1]$. Define $V_{j}%
:=(U_{j}^{j})_{\beta_{j}}=(1-\beta_{j})\operatorname*{Id}+\beta_{j}U_{j}^{j}$,
$j\in J$. Clearly, $U_{j}^{j}=(V_{j})_{\beta_{j}^{-1}}$, $j\in J$. Define
$V:\mathcal{H}\rightarrow\mathcal{H}$ by $V(x)=(V_{1}(x_{1}),V_{2}%
(x_{2}),...,V_{n}(x_{n}))$, $x\in\mathcal{H}$. The operators $V_{j}$, $j\in
J$, are NE as convex combinations of NE operators $U_{j}^{j}$ and
$\operatorname*{Id}$. Again, by Fact \ref{f2}, $V$ is CW-NE. We have
\begin{equation}
(T(x))_{j}=(U_{j}^{j})_{\alpha_{j}}(x_{j})=((V_{j})_{\beta_{j}^{-1}}%
)_{\alpha_{j}}(x_{j})=(V_{j})_{\alpha}(x_{j})=(V_{\alpha}(x))_{j}\text{.}%
\end{equation}
Thus, $T=V_{\alpha}$, where $V$ is CW-NE and $\alpha\in(0,1)$.

In a similar way one can prove that we can set a common constant $\lambda
\in\lbrack0,2]$ and a common CW-FNE operator $U$ in the definition of a
CW-RFNE operator $T$. This yields that a CW-RFNE operator is RFNE.
\end{proof}

\begin{fact}
\label{f3}Let $j\in J$ and let $T:\mathcal{H}\rightarrow\mathcal{H}$. The
following conditions are equivalent:

\begin{enumerate}
\item[$\mathrm{(i)}$] $T$ is $j$-FNE;

\item[$\mathrm{(ii)}$] $T_{\lambda}$ is $j$-NE for all $\lambda\in\lbrack0,2]$;

\item[$\mathrm{(iii)}$] $T$ has the form $T=\frac{1}{2}(\operatorname*{Id}+S)$
for some $j$-NE operator $S$, i.e., $T$ is $(\frac{1}{2},j)$-AV;

\item[$\mathrm{(iv)}$] $\operatorname*{Id}-T$ is $j$-FNE;

\item[$\mathrm{(v)}$] The following inequality holds for all $x,y\in
\mathcal{H}$
\begin{equation}
\Vert(T(x))_{j}-(T(y))_{j}\Vert_{j}^{2}\leq\Vert x_{j}-y_{j}\Vert_{j}%
^{2}-\Vert(x_{j}-(T(x))_{j})-(y_{j}-(T(y))_{j})\Vert_{j}^{2}.
\end{equation}

\end{enumerate}
\end{fact}

\begin{proof}
Similar to \cite[Theorem 2.2.10]{cegielski-book}.
\end{proof}

\bigskip

Denote $I:=\{1,2,...,m\}$, where $m\in\mathbb{N}$.

\begin{corollary}
Let $j\in J$, $T_{i}:\mathcal{H}\rightarrow\mathcal{H}$ be $j$-FNE, for all
$i\in I$ and $w=(w_{1},w_{2},...,w_{m})\in\Delta_{m}:=\{w\in%
\mathbb{R}
^{m}\mid w\geq0,$ $\sum_{i=1}^{m}w_{i}=1\}.$ Then the operator $T:\mathcal{H}%
\rightarrow\mathcal{H},$ defined by $T(x):=\sum_{i=1}^{m}w_{i}T_{i}(x),$ is
$j$-FNE.
\end{corollary}

\begin{proof}
Similar to \cite[Corollary 2.2.20]{cegielski-book}.
\end{proof}

\bigskip{}

For a $j$-NE operator $T_{i}:\mathcal{H}\rightarrow\mathcal{H},$ the operators
$T_{i}^{j}:\mathcal{H}_{j}\rightarrow\mathcal{H}_{j}$ are defined in a similar
way as in (9), i.e.,
\begin{equation}
T_{i}^{j}(x_{j}):=(T_{i}(x))_{j}\text{, }i\in I,\text{ }j\in J.\label{e-Tijxj}%
\end{equation}

\begin{fact}
\label{f4}Let $j\in J$, $w_{ij}\geq0$, $i\in I$, $w_{\cdot j}:=\sum_{i=1}%
^{m}w_{ij}>0$. If $T_{i}:\mathcal{H}\rightarrow\mathcal{H}$, $i\in I$, are
$j$-FNE, then the operator $T:\mathcal{H}\rightarrow\mathcal{H}$ defined by
its components $T^{j}:\mathcal{H}_{j}\rightarrow\mathcal{H}_{j}$,%
\begin{equation}
T^{j}(x_{j}):=x_{j}+\lambda\sum_{i=1}^{m}w_{ij}(T_{i}^{j}(x_{j})-x_{j}%
)\text{,} \label{e-Tjxj}%
\end{equation}
where $\lambda\in\lbrack0,2/w_{\cdot j}]$, is $(\lambda w_{\cdot j},j)$-RFNE.
\end{fact}

\begin{proof}
Clearly, $T_{i}^{j}$ are FNEs, $i\in I$. Define an operator $S_{j}%
:\mathcal{H}_{j}\rightarrow\mathcal{H}_{j}$ by
\begin{equation}
S_{j}(x_{j}):=\sum_{i=1}^{m}v_{i}T_{i}^{j}(x_{j})\text{,} \label{e-Sjxj}%
\end{equation}
where $v_{i}:=\frac{{\textstyle w_{ij}}}{{\textstyle w_{\cdot j}}}$. Since
$\sum_{i=1}^{m}v_{i}=1$, the operator $S_{j}$ is FNE, see \cite[Corollary
2.2.20]{cegielski-book}, and we have that $T^{j}=(S_{j})_{\lambda w_{\cdot j}%
}$, the $(\lambda w_{\cdot j})$-relaxation of $S_{j},$ i.e., $T^{j}$ is
$(\lambda w_{\cdot j},j)$-RFNE.
\end{proof}

\begin{corollary}
Let $w_{ij}\geq0$, $w_{\cdot j}:=\sum_{i=1}^{m}w_{ij}>0,$ $i\in I,$ $j\in J$,
and $w:=\max_{j\in J}w_{\cdot j}$. If $T_{i}:\mathcal{H}\rightarrow
\mathcal{H}$, $i\in I$, are CW-FNE, then the operator $T$ defined by
\begin{equation}
(T(x))_{j}:=x_{j}+\lambda\sum_{i=1}^{m}w_{ij}((T_{i}(x_{j}))_{j}%
-x_{j})\text{,} \label{e-Txj2}%
\end{equation}
$j\in J$, where $\lambda\in\lbrack0,2/w]$, is $(\lambda w)$-CW-RFNE.
\end{corollary}

\begin{proof}
By Facts \ref{f2} and \ref{f4}, $T^{j}$ is $\lambda w_{\cdot j}$-RFNE, $j\in
J$. Since $0\leq\lambda w_{\cdot j}\leq\lambda w\leq2$, $T^{j}$ is $\lambda
w$-RFNE, $j\in J$. Noting that $T(x)=(T^{1}(x_{1}),T^{2}(x_{2}),...,T^{n}%
(x_{n}))$, tells us that $T$ is $\lambda w$-RFNE.
\end{proof}

\begin{example}%
\rm\
Consider a consistent system of linear equations $Ax=b$, where $A$ is an
$m\times n$ matrix with rows $a_{i}\in%
\mathbb{R}%
^{n}$ and $b\in%
\mathbb{R}%
^{m}$. Consider the hyperplanes $H_{i}:=\{x\in%
\mathbb{R}%
^{n}\mid\left\langle a_{i},x\right\rangle =b_{i}\}$, $i\in I$. A special case
of an operator $T$ defined by (\ref{e-Txj2}), occurs when $T_{i}:=P_{H_{i}}$,
the metric projection onto $H_{i},$ $w_{ij}:=w_{i}/s_{j}$, $i\in I$, $i\in J$,
with $w=(w_{1},w_{2},...,w_{m})\in\Delta_{m}$ and $s_{j}$ being the number of
nonzero elements in the $j$-th column of $A.$ This special case, introduced
and investigated in \cite[Equation (1.11)]{CEHN08}, is our motivating example,
discussed in Section \ref{sec:nonlin-drop} below.
\end{example}

\subsection{Componental cutter, quasi-nonexpansive, strongly
quasi-nonexpansive and strictly quasi-nonexpansive operators\label{sec:CW-QNE}%
}

Let $T:\mathcal{H}\rightarrow\mathcal{H}$ be defined by (\ref{e-T}). For $j\in
J$ we define the ``$j$-th fixed point set of $T$''\ by
\begin{equation}
\operatorname*{Fix}{}^{j}T:=\{z\in\mathcal{H}\mid(T(z))_{j}=z_{j}%
\}\text{.\label{eq:j-fix-point}}%
\end{equation}
Clearly, $\operatorname*{Fix}T=\bigcap_{j\in J}\operatorname*{Fix}{}^{j}T$.

Bearing in mind the well-known definitions of operators that are
\textit{cutter, quasi-nonexpansive} (QNE), \textit{strongly
quasi-nonexpansive} (SQNE), or \textit{strictly quasi-nonexpansive} (sQNE),
see, e.g., \cite{cegielski-book}, we introduce the following additional new
definitions of componental operators.

\begin{definition}
\label{d2}%
\rm\
Let $j\in J$ and let $T:\mathcal{H}\rightarrow\mathcal{H}$ be an operator with
$\operatorname*{Fix}{}^{j}T\neq\emptyset$. We say that $T$ is:

\begin{enumerate}
\item[(i)] a $j$\textit{-cutter}, if for all $x\in H$ and $z\in
\operatorname*{Fix}{}^{j}T$ it holds that
\begin{equation}
\langle x_{j}-(T(x))_{j},z_{j}-(T(x))_{j}\rangle_{j}\leq0\text{;}
\label{e-j-cut}%
\end{equation}

\item[(ii)] $j$\textit{-quasi-nonexpansive} ($j$-QNE), if for all $x\in H$ and
$z\in\operatorname*{Fix}{}^{j}T$ it holds that
\begin{equation}
\Vert(T(x))_{j}-z_{j}\Vert_{j}\leq\Vert x_{j}-z_{j}\Vert_{j}\text{;}
\label{e-j-QNE}%
\end{equation}

\item[(iii)] $j$\textit{-strongly quasi-nonexpansive} ($j$-SQNE), if there is
a constant $\rho_{j}>0$ such that for all $x\in H$ and $z\in
\operatorname*{Fix}{}^{j}T$ it holds that
\begin{equation}
\Vert(T(x))_{j}-z_{j}\Vert_{j}^{2}\leq\Vert x_{j}-z_{j}\Vert_{j}^{2}-\rho
_{j}\Vert(T(x))_{j}-x_{j}\Vert_{j}^{2}\text{;} \label{e-j-SQNE}%
\end{equation}

\item[(iv)] $j$\textit{-strictly quasi-nonexpansive} ($j$-sQNE), if for all
$x\notin\operatorname*{Fix}{}^{j}T$ and $z\in\operatorname*{Fix}{}^{j}T$ it
holds that
\begin{equation}
\Vert(T(x))_{j}-z_{j}\Vert_{j}<\Vert x_{j}-z_{j}\Vert_{j}\text{;}
\label{e-j-sQNE}%
\end{equation}

\item[(v)] a \textit{component-wise cutter} (CW-QNE, CW-SQNE, CW-sQNE ) if $T$
is $j$-cutter ($j$-QNE, $j$-SQNE, $j$-sQNE) for all $j\in J$.
\end{enumerate}

Alternatively, if we want to emphasize the constant $\rho_{j}$ in (iii) or the
vector $r:=(\rho_{1},\rho_{2},...,\rho_{n})$ in (v) explicitly, then we say
that $T$ is $(\rho_{j},j)$-SQNE or $T$ is $r$-CW-SQNE, respectively.
\end{definition}

Note that if $T$ is $(\rho_{j},j)$-SQNE and $\rho_{j}\geq\rho>0$ then $T$ is
$(\rho,j)$-SQNE. Thus, one can set the constant $\rho$ in (v), which does not
depend on $j$, e.g., $\rho:=\min_{j\in J}\rho_{j}$.

\begin{example}
\label{ex1}%
\rm\
Let $U^{j}:\mathcal{H}_{j}\rightarrow\mathcal{H}_{j}$, $j\in J$, and define
$U:\mathcal{H}\rightarrow\mathcal{H}$ by
\begin{equation}
U(x):=(U^{1}(x_{1}),U^{2}(x_{2}),...,U^{n}(x_{n}))\text{.} \label{e-Ux}%
\end{equation}
Clearly, $\operatorname*{Fix}U=\prod_{j=1}^{n}\operatorname*{Fix}U^{j}$. By
definition, $(U(x))_{j}=U^{j}(x_{j})$. Thus, it follows from Definitions
\ref{d1} and \ref{d2} that $U^{j}$ is NE (a cutter, QNE, SQNE), if and only if
$U$ is $j$-NE ($j$-cutter, $j$-QNE, $j$-SQNE), $j\in J$. Consequently, $U^{j}$
are NE (a cutter, QNE, SQNE) for all $j\in J$, if and only if $U$ is CW-NE
(CW-cutter, CW-QNE, CW-SQNE).
\end{example}

\begin{example}
\label{ex2}%
\rm\
Let $S_{j}:\mathcal{H}\rightarrow\mathcal{H}$, $j\in J$, and define
$S:\mathcal{H}\rightarrow\mathcal{H}$ by
\begin{equation}
S(x):=((S_{1}(x))_{1},(S_{2}(x))_{2},...,(S_{n}(x))_{n})\text{.}%
\end{equation}
Then $(S(x))_{j}=(S_{j}(x))_{j}$, $j\in J$. Consequently, $S$ is a CW-cutter
($r$-CW-SQNE, where $r:=(\rho_{1},\rho_{2},...,\rho_{n})$ with $\rho_{j}>0$,
$j\in J)$ if and only if $S_{j}$ is a $j$-cutter ($(\rho_{j},j)$-SQNE) for all
$j\in J$.\medskip
\end{example}

For $j\in J$ and $T:\mathcal{H}\rightarrow\mathcal{H}$ define
\begin{equation}
F^{j}(T):=\{z\in\mathcal{H}\mid\Vert(T(x))_{j}-z_{j}\Vert_{j}\leq\Vert
x_{j}-z_{j}\Vert_{j}\text{ for all }x\in\mathcal{H}\} \label{e-(F)j}%
\end{equation}
and
\begin{equation}
F_{j}(T):=\{z_{j}\in\mathcal{H}_{j}\mid\Vert(T(x))_{j}-z_{j}\Vert_{j}\leq\Vert
x_{j}-z_{j}\Vert_{j}\text{ for all }x\in\mathcal{H}\}\text{.} \label{e-Fj}%
\end{equation}
Clearly,
\begin{equation}
F_{j}(T)=\bigcap_{x\in\mathcal{H}}\{z_{j}\in\mathcal{H}_{j}\mid\Vert
(T(x))_{j}-z_{j}\Vert_{j}\leq\Vert x_{j}-z_{j}\Vert_{j}\}\text{,}
\label{e-Fj=00003D}%
\end{equation}
thus, $F_{j}$ is a closed convex subset as intersection of closed half-spaces.

For a $j$-QNE operator $T$, where $j\in J$, the property expressed in Fact
\ref{f1} is not true in general. Thus, contrary to an NE operator, a CW-QNE
operator cannot be decomposed. Nevertheless, the fixed point set of a CW-QNE
operator is a Cartesian product of some sets.\medskip

\begin{fact}
\label{f7}Let $j\in J$.

\begin{enumerate}
\item[$\mathrm{(i)}$] The following inclusion holds
\begin{equation}
F^{j}(T)\subseteq\operatorname*{Fix}{}^{j}T\text{.} \label{e-Fj-Fixj}%
\end{equation}

\item[$\mathrm{(ii)}$] If $T$ is $j$-QNE then the converse inclusion is also
true. Consequently, for a CW-QNE operator $T$ we have
\begin{equation}
\operatorname*{Fix}T=\bigcap_{j=1}^{n}\operatorname*{Fix}{}^{j}T=\bigcap
_{j=1}^{n}F^{j}(T)=\prod_{j=1}^{n}F_{j}(T)\text{.} \label{e-Fix-F}%
\end{equation}

\item[$\mathrm{(iii)}$] If $T:\mathcal{H}\rightarrow\mathcal{H}$ is a
CW-cutter (CW-QNE, CW-SQNE, CW-sQNE), then $T$ is a cutter (QNE, SQNE, sQNE).
\end{enumerate}
\end{fact}

\begin{proof}
(i) If $F^{j}(T)=\emptyset$ then the inclusion in (\ref{e-Fj-Fixj}) is clear.
Let now $F^{j}(T)\neq\emptyset$ and $z\in F^{j}(T)$. If we take $x=z$ in
(\ref{e-(F)j}) then we obtain
\begin{equation}
\Vert(T(z))_{j}-z_{j}\Vert_{j}\leq\Vert z_{j}-z_{j}\Vert_{j}=0\text{,}%
\end{equation}
i.e., $z\in\operatorname*{Fix}{}^{j}T$.

(ii) Suppose that $T$ is $j$-QNE and let $z\in\operatorname*{Fix}{}^{j}T$.
Then, by definition, $z\in F^{j}(T)$. This together with (i) implies the
second equality in (\ref{e-Fix-F}) if $T$ is CW-QNE. The first and the last
equalities in (\ref{e-Fix-F}) are obvious.

(iii) Follows from the definition of the inner product in $\mathcal{H}$ and
from the definition of a CW-cutter (CW-QNE, CW-SQNE operator).
\end{proof}

\bigskip{}

If we consider a relaxation of an operator $T:\mathcal{H}\rightarrow
\mathcal{H}$ we can associate different relaxation parameters $\lambda_{j}$
with various components $(T(\cdot))_{j}$. Let $y:=(\lambda_{1},\lambda
_{2},...,\lambda_{n})\in%
\mathbb{R}
_{+}^{n}$. We say that the operator $T_{y},$ defined by,
\begin{equation}
(T_{y}(x))_{j}:=x_{j}+\lambda_{j}((T(x))_{j}-x_{j})\text{, }j\in J\text{,}%
\end{equation}
is the $y$\textit{-CW-relaxation} of $T$ or, in short, a
\textit{CW-relaxation} of $T$. Clearly, if $\lambda_{j}=\lambda$ for all $j\in
J$ then the notion of the $y$-CW-relaxation of $T$ coincides with the notion
of the $\lambda$-relaxation of $T$.\medskip

\begin{fact}
\label{f5}Let $T:\mathcal{H}\rightarrow\mathcal{H}$ with $\operatorname*{Fix}%
T\neq\emptyset$, let $y=(\lambda_{1},\lambda_{2},...,\lambda_{n})\in(0,2]^{n}%
$. If $T$ is a CW-cutter then its $y$-relaxation $T_{y}$ is $r$-CW-SQNE, where
$r:=(\rho_{1},\rho_{2},...,\rho_{n})$ with $\rho_{j}=(2-\lambda_{j}%
)/\lambda_{j}$, $j\in J$.
\end{fact}

\begin{proof}
Suppose that $T$ is a CW-cutter and let $z\in\operatorname*{Fix}T$. For any
$j\in J$ we have
\begin{align}
\Vert(T_{y}(x))_{j}-z_{j}\Vert_{j}^{2}  &  =\Vert x_{j}+\lambda_{j}%
((T(x))_{j}-x_{j})-z_{j}\Vert_{j}^{2}\nonumber\\
&  =\Vert x_{j}-z_{j}\Vert_{j}^{2}+\lambda_{j}^{2}\Vert(T(x))_{j}-x_{j}%
\Vert_{j}^{2}\nonumber\\
&  +2\lambda_{j}\langle x_{j}-(T(x))_{j},z_{j}-(T(x))_{j}\rangle_{j}%
-2\lambda_{j}\Vert(T(x))_{j}-x_{j}\Vert_{j}^{2}\nonumber\\
&  \leq\Vert x_{j}-z_{j}\Vert_{j}^{2}-\lambda_{j}(2-\lambda_{j})\Vert
(T(x))_{j}-x_{j}\Vert_{j}^{2}\nonumber\\
&  =\Vert x_{j}-z_{j}\Vert_{j}^{2}-\frac{2-\lambda_{j}}{\lambda_{j}}%
\Vert(T_{\lambda}(x))_{j}-x_{j}\Vert_{j}^{2}\text{,}%
\end{align}
which means that $T_{y}$ is $r$-CW-SQNE.
\end{proof}

\begin{fact}
\label{f6}Let $j\in J$. A $j$-NE operator having a fixed point is $j$-QNE.
Consequently, a CW-NE operator having a fixed point is CW-QNE.
\end{fact}

\begin{proof}
Let $T$ be $j$-NE and let $z\in\operatorname*{Fix}{}^{j}T$. Then
$(T(z))_{j}=z_{j}$ and, for all $x\in\mathcal{H},$ we have
\begin{equation}
\Vert(T(x))_{j}-z_{j}\Vert_{j}=\Vert(T(x))_{j}-(T(z))_{j}\Vert_{j}\leq\Vert
x_{j}-z_{j}\Vert_{j}\text{.}%
\end{equation}
This means that $T$ is $j$-QNE.
\end{proof}

\begin{fact}
\label{f-8a}Let $j\in J$, let $T_{i}:\mathcal{H}\rightarrow\mathcal{H}$ be
$j$-sQNE, $i\in I$, with $\bigcap_{i\in I}F^{j}(T_{i})\neq\emptyset$, and
$T:=\sum_{i=1}^{m}w_{i}T_{i}$, where $w\in\operatorname*{ri}\Delta_{m}$
($\operatorname*{ri}$ is the relative interior). Then
\begin{equation}
F^{j}(T)=\bigcap_{i\in I}F^{j}(T_{i}) \label{e-FjTi}%
\end{equation}
and $T$ is $j$-sQNE.
\end{fact}

\begin{proof}
To prove the inclusion $\supseteq$ in (\ref{e-FjTi}) , let $z\in\bigcap_{i\in
I}F^{j}(T_{i})$ and $x\in\mathcal{H}$ be arbitrary. The convexity of the norm
and the assumption that $T_{i}$ are $j$-QNE, $i\in I$, yield
\begin{equation}
\Vert(T(x))_{j}-z_{j}\Vert_{j}\leq\sum_{i=1}^{m}w_{i}\Vert(T_{i}(x))_{j}%
-z_{j}\Vert_{j}\leq\sum_{i=1}^{m}w_{i}\Vert x_{j}-z_{j}\Vert_{j}=\Vert
x_{j}-z_{j}\Vert_{j}\text{,}%
\end{equation}
which shows that $z\in F^{j}(T)$.

To prove the inclusion $\subseteq$ in (\ref{e-FjTi}) we observe that the
inclusion is clear if $\bigcap_{i\in I}F^{j}(T_{i})=\mathcal{H}$. Suppose the
opposite and let $x\notin\bigcap_{i\in I}F^{j}(T_{i})$ and $z\in\bigcap_{i\in
I}F^{j}(T_{i})$. The convexity of the norm and the assumption that $T_{i}$ are
$j$-sQNE, $i\in I$, yield
\begin{equation}
\Vert(T(x))_{j}-z_{j}\Vert_{j}\leq\sum_{i=1}^{m}w_{i}\Vert(T_{i}(x))_{j}%
-z_{j}\Vert_{j}<\sum_{i=1}^{m}w_{i}\Vert x_{j}-z_{j}\Vert_{j}=\Vert
x_{j}-z_{j}\Vert_{j}\text{,}\label{e-Txj-zj}%
\end{equation}
because $\Vert(T_{i_{0}}(x))_{j}-z_{j}\Vert_{j}<\Vert x_{j}-z_{j}\Vert_{j}$
for some $i_{0}$, and $w_{i_{0}}>0$. Now it is clear that $x\notin
\operatorname*{Fix}{}^{j}T=F^{j}(T)$, because otherwise,
\begin{equation}
\Vert x_{j}-z_{j}\Vert_{j}=\Vert(T(x))_{j}-z_{j}\Vert_{j}<\Vert x_{j}%
-z_{j}\Vert_{j}%
\end{equation}
which would lead to a contradiction.\medskip
\end{proof}

\begin{fact}
\label{f8}Let $j\in J$, $T_{i}:\mathcal{H}\rightarrow\mathcal{H}$ be
$j$-cutters having a common fixed point, $i\in I$ and $w\in\operatorname*{ri}%
\Delta_{m}$. Then the operator $T:\mathcal{H}\rightarrow\mathcal{H}$ defined
by $T(x):=\sum_{i=1}^{m}w_{i}T_{i}(x)$ is a $j$-cutter.
\end{fact}

\begin{proof}
Similar to \cite[Corollary 2.1.49]{cegielski-book}.
\end{proof}

\begin{corollary}
Let $T_{i}:\mathcal{H}\rightarrow\mathcal{H}$ be CW-cutters having a common
fixed point, $i\in I$, and $w\in\Delta_{m}$. Then the operator $T:\mathcal{H}%
\rightarrow\mathcal{H}$ defined by $T(x)=\sum_{i=1}^{m}w_{i}T_{i}(x)$ is a
CW-cutter. If, moreover, $w_{i}>0$, $i\in I$, then $\operatorname*{Fix}%
T=\bigcap_{i\in I}\operatorname*{Fix}T_{i}$.\medskip
\end{corollary}

\begin{fact}
\label{f9}Let $w_{ij}\geq0$, $i\in I$, $w_{\cdot j}:=\sum_{i=1}^{m}w_{ij}>0$,
$j\in J$. If $T_{i}:\mathcal{H}\rightarrow\mathcal{H}$, $i\in I$, are
CW-cutters having a common fixed point then the operator $T:\mathcal{H}%
\rightarrow\mathcal{H},$ defined by
\begin{equation}
(T(x))_{j}:=x_{j}+\lambda_{j}\sum_{i=1}^{m}w_{ij}((T_{i}(x))_{j}%
-x_{j})\text{,} \label{eq:nonlin-drop}%
\end{equation}
$j\in J$, where $\lambda_{j}\in(0,2/w_{\cdot j})$, is $r$-CW-SQNE, where
$r:=(\rho_{1},\rho_{2},...,\rho_{n})$ with $\rho_{j}=(2-\lambda_{j}%
)/\lambda_{j},j\in J$.
\end{fact}

\begin{proof}
Define an operator $S_{j}:\mathcal{H}\rightarrow\mathcal{H}$ by
\begin{equation}
S_{j}(x):=\sum_{i=1}^{m}\frac{w_{ij}}{w_{\cdot j}}T_{i}(x)\text{, }j\in J.
\end{equation}
By Fact \ref{f8}, $S_{j}$ is a $j$-cutter, $j\in J$. Define, for
$x\in\mathcal{H}$,
\begin{equation}
S(x)=((S_{1}(x))_{1},(S_{2}(x))_{2},...,(S_{n}(x))_{n})\text{.}%
\end{equation}
The operator $S:\mathcal{H}\rightarrow\mathcal{H}$ is a CW-cutter (see Example
\ref{ex2}). We have
\begin{equation}
(S(x))_{j}=(S_{j}(x))_{j}=\sum_{i=1}^{m}\frac{w_{ij}}{w_{\cdot j}}%
(T_{i}(x))_{j}\text{,}%
\end{equation}
and
\begin{equation}
(T(x))_{j}=x_{j}+\lambda w_{\cdot j}((S(x))_{j}-x_{j})\text{.} \label{e-Txj3}%
\end{equation}
Denote $y=(\lambda w_{\cdot1},\lambda w_{\cdot2},...,\lambda w_{\cdot n})$.
Equality (\ref{e-Txj3}) means that $T=S_{y}$, i.e., $T$ is a $y$-CW-relaxation
of a CW-cutter $S$. Thus, Fact \ref{f5} yields that the operator $T$ is
$r$-CW-SQNE, where $r:=(\rho_{1},\rho_{2},...,\rho_{n})$ with $\rho
_{j}=(2-\lambda_{j})/\lambda_{j},j\in J$.
\end{proof}

\subsection{Componental contractions\label{subsec:contract}}

\begin{definition}%
\rm\
An operator $T:\mathcal{H}\rightarrow\mathcal{H}$ is a $j$%
\textit{-contraction} ($j$-CONT) if, for some $j\in J,$
\begin{equation}
\Vert(T(x))_{j}-(T(y))_{j}\Vert_{j}\leq\alpha_{j}\Vert x_{j}-y_{j}\Vert
_{j},\text{ for all }x,y\in\mathcal{H},
\end{equation}
with an $\alpha_{j}\in\lbrack0,1)$. Alternatively, if we want to emphasize the
constant $\alpha_{j}$ explicitly, then we say that $T$ is an $(\alpha_{j},j)$-contraction.
\end{definition}

If $T$ is an $(\alpha_{j},j)$-contraction for all $j\in J$ then it is an
$\alpha$-contraction with $\alpha:=\max_{1\leq j\leq n}\alpha_{j}.$ That the
opposite is not true follows from the counter example $T(x_{1},x_{2}):=\left(
{\textstyle \frac{1}{2}x_{2},{\textstyle \frac{1}{2}x_{1}}}\right)  $ which is
an $\alpha$-contraction, with, e.g., $\alpha=0.5$, since
\begin{equation}
\left\Vert T(x)-T(y)\right\Vert =\frac{1}{2}\left\Vert x-y\right\Vert .
\end{equation}
However, for the points $(0,0)$ and $(1,5)$ there does not exist a real number
$\alpha_{1}\in\lbrack0,1)$ with which $T$ is $(\alpha_{1},1)$-contractive.

Following the Banach fixed point theorem, as reformulated in Berinde's book
\cite[Theorem 2.1]{berinde-book}, we formulate a componental contraction
mapping principle as follows.

\begin{theorem}
\label{thm:comp-contract}Let $j\in J$ and let $T:\mathcal{H}\rightarrow
\mathcal{H}$ be an $(\alpha_{j},j)$-contraction. Then

\begin{enumerate}
\item[$\mathrm{(i)}$] The $j$-th componental fixed point set of $T$ is
nonempty, i.e., $\operatorname*{Fix}{}^{j}T:=\{x\in\mathcal{H}\mid
(T(x))_{j}=x_{j}\}\neq\emptyset$ and the variable $x_{j}$ is unique for all
$x\in\operatorname*{Fix}{}^{j}T,$ henceforth denoted as $x_{j}=x_{j\ast}$.

\item[$\mathrm{(ii)}$] The sequence $\{x_{j}^{k}\}_{k=0}^{\infty}$ of the
$j$-th components of any sequence $\{x^{k}\}_{k=0}^{\infty},$ generated by the
Picard iteration $x^{k+1}=T(x^{k})$ associated with $T,$ converges for any
initial point $x^{0}\in\mathcal{H},$ and
\begin{equation}
\lim_{k\rightarrow\infty}x_{j}^{k}=x_{j\ast}.
\end{equation}

\item[$\mathrm{(iii)}$] The following a priori and a posteriori error
estimates hold:
\begin{equation}
\left\Vert x_{j}^{k}-x_{j\ast}\right\Vert _{j}\leq\frac{(\alpha_{j})^{k}%
}{1-\alpha_{j}}\left\Vert x_{j}^{0}-x_{j}^{1}\right\Vert _{j},\text{ for all
}k=1,2,\ldots,
\end{equation}
\begin{equation}
\left\Vert x_{j}^{k}-x_{j\ast}\right\Vert _{j}\leq\frac{\alpha_{j}}%
{1-\alpha_{j}}\left\Vert x_{j}^{k-1}-x_{j}^{k}\right\Vert _{j},\text{ for all
}k=1,2,\ldots.
\end{equation}

\item[$\mathrm{(iv)}$] The rate of convergence of the sequence $\{x_{j}%
^{k}\}_{k=0}^{\infty},$ in $\mathrm{(ii)}$ above, is given by
\begin{equation}
\left\Vert x_{j}^{k}-x_{j\ast}\right\Vert _{j}\leq\alpha_{j}\left\Vert
x_{j}^{k-1}-x_{j\ast}\right\Vert _{j}\leq(\alpha_{j})^{k}\left\Vert x_{j}%
^{0}-x_{j\ast}\right\Vert _{j},\text{ for all }k=1,2,\ldots.
\end{equation}

\end{enumerate}
\end{theorem}

\begin{proof}
This can be proved exactly along the lines of the proof in \cite[Theorem
2.1]{berinde-book}. Alternatively, one can introduce an operator $U$
$:\mathcal{H}\rightarrow\mathcal{H}$ by $U(x):=(0,...,0,T_{j}(x),0,...,0)$ and
apply \cite[Theorem 2.1]{berinde-book} to it.
\end{proof}

\bigskip{}

To justify Theorem \ref{thm:comp-contract} we build an example of an operator
which is not an $\alpha$-contraction but is an $\alpha_{j}$-contractions for
some indices $j\in J,$ but not all. The original Banach fixed point theorem
would not apply to them but our theorem would.

\begin{example}%
\rm\
Let $T:%
\mathbb{R}
^{2}\rightarrow%
\mathbb{R}
^{2}$ be defined by
\begin{equation}
T(x_{1},x_{2})=\left(
\begin{array}
[c]{c}%
T_{1}(x_{1},x_{2})\\
T_{2}(x_{1},x_{2})
\end{array}
\right)  :=\left(
\begin{array}
[c]{c}%
\frac{x_{1}}{2}+3\\
8x_{2}%
\end{array}
\right)  .
\end{equation}
This $T$ is an $\alpha_{1}$-contraction with, e.g., $\alpha_{1}=1/2$ because
\begin{equation}
\left\vert T_{1}(x_{1},x_{2})-T_{1}(y_{1},y_{2})\right\vert =\frac{1}%
{2}\left\vert x_{1}-y_{1}\right\vert .
\end{equation}
But $T$ is not an $\alpha$-contraction because
\begin{equation}
\left\Vert T(x)-T(y)\right\Vert ^{2}=\frac{1}{4}(x_{1}-y_{1})^{2}%
+64(x_{2}-y_{2})^{2}%
\end{equation}
and there is no $\alpha\in\lbrack0,1)$ for which
\begin{equation}
\frac{1}{4}(x_{1}-y_{1})^{2}+64(x_{2}-y_{2})^{2}\leq\alpha^{2}\left(
(x_{1}-y_{1})^{2}+(x_{2}-y_{2})^{2}\right)
\end{equation}
for all $x,y\in%
\mathbb{R}
^{2}.$
\end{example}

\section{Regularity of component-wise quasi-nonexpansive
operators\label{sec:regularity}}

The notions of asymptotic regularity of sequences and operators play a central
role in fixed point theory, see, e.g., \cite{bau-com-book} or
\cite{cegielski-book}. We define next a notion of componental regularity.

\begin{definition}%
\rm\
Let $T:\mathcal{H}\rightarrow\mathcal{H}$ be a QNE operator and let $j\in J$.

\begin{enumerate}
\item[(i)] We say that $T$ is $j$\textit{-weakly regular}$\;$($j$-WR) if, for
any sequence $\{x^{k}\}_{k=0}^{\infty}\subseteq\mathcal{H}$ and some
$y\in\mathcal{H}$,
\begin{equation}
x_{j}^{k}\rightharpoonup y_{j}\text{ and }\lim_{k\rightarrow\infty}%
\Vert(T(x^{k}))_{j}-x_{j}^{k}\Vert_{j}=0\Longrightarrow y\in
\operatorname*{Fix}{}^{j}T. \label{eq:WRj}%
\end{equation}

\item[(ii)] If $T$ is $j$-weakly regular for all $j\in J$, then we say that
$T$ is \textit{CW-weakly regular} (CW-WR).

\item[(iii)] If (\ref{eq:WRj}) holds after all $j$ indices are removed from it
then we say that $T$ is weakly regular$\;$(WR) (cf. \cite[Definition
3.1]{CRZ18}).
\end{enumerate}
\end{definition}

The weak regularity of an operator $T$ means that $T-\operatorname*{Id}$ is
demi-closed at $0$ (cf. \cite{Opi67}).\medskip

\begin{fact}
\label{f-CW-WR}A QNE operator $T:\mathcal{H}\rightarrow\mathcal{H}$ is CW-WR
if and only if $T$ is WR.
\end{fact}

\begin{proof}
Suppose that $T$ is CW-WR. Let $x^{k}\rightharpoonup y$ and $\lim
_{k\rightarrow\infty}\Vert T(x^{k})-x^{k}\Vert=0$. Then $x_{j}^{k}%
\rightharpoonup y_{j}$ and $\lim_{k\rightarrow\infty}\Vert(T(x^{k}))_{j}%
-x_{j}^{k}\Vert_{j}=0$ for all $j\in J$. Since $T$ is $j$-WR, $y\in
\operatorname*{Fix}{}^{j}T$, for all $j\in J$, thus, $y\in\bigcap_{j\in
J}\operatorname*{Fix}{}^{j}T=\operatorname*{Fix}T$ which means that $T$ is WR.
The converse can be proved similarly.
\end{proof}

\begin{theorem}
\label{f-WC}Let $T:\mathcal{H}\rightarrow\mathcal{H}$ be an $r$-CW-SQNE and
CW-WR operator, where $r=(\rho_{1},\rho_{2},...,\rho_{n})$ with $\rho
_{j}>0,j\in J$, and let the sequence $\{x^{k}\}_{k=0}^{\infty}$ be generated
by the iteration
\begin{equation}
x_{j}^{k+1}=(T(x^{k}))_{j},\text{\ }k\geq0,\label{eq:iter-process}%
\end{equation}
$j\in J$ , where $x^{0}\in\mathcal{H}$ is arbitrary. Then $\{x^{k}%
\}_{k=0}^{\infty}$ converges weakly to some $x^{\ast}\in\operatorname*{Fix}T$.
\end{theorem}

\begin{proof}
By assumption, $T$ is $\rho_{j}$-SQNE, $j\in J$, i.e., for any $z=(z_{1}%
,z_{2},...,z_{n})\in\operatorname*{Fix}T=\bigcap_{j\in J}\operatorname*{Fix}%
{}^{j}T$ it holds%

\begin{equation}
\Vert x_{j}^{k+1}-z_{j}\Vert_{j}^{2}=\Vert(T(x^{k}))_{j}-z_{j}\Vert_{j}%
^{2}\leq\Vert x_{j}^{k}-z_{j}\Vert_{j}^{2}-\rho_{j}\Vert(T(x^{k}))_{j}%
-x_{j}^{k}\Vert_{j}^{2}\text{.}%
\end{equation}

That ${\operatorname*{Fix}}T$ is nonempty follows from Definition \ref{d2}.
Using standard arguments this leads to the boundedness of $x_{j}^{k}$, $j\in
J$ and to $\Vert(T(x^{k}))_{j}-x_{j}^{k}\Vert_{j}\rightarrow0$ as
$k\rightarrow\infty$. Thus, there exists a subsequence $\{x_{j}^{n_{k}%
}\}_{k=0}^{\infty}$ which converges weakly to some $x_{j}^{\ast}$, $j\in J$.
Let $x^{\ast}=(x_{1}^{\ast},x_{2}^{\ast},...,x_{n}^{\ast})$. Because $T$ is
$j$-weakly regular, we have $x^{\ast}\in\operatorname*{Fix}{}^{j}T$, $j\in J$.
This and Fact \ref{f7}(ii) yields $x^{\ast}\in\operatorname*{Fix}T$. By
standard arguments, like in \cite[Theorem 2.16(ii)]{bb96}, the whole sequence
$\{x^{k}\}_{k=0}^{\infty}$ converges weakly to $x^{\ast}$.
\end{proof}

\section{Nonlinear DROP as a simultaneous projection method with
component-wise weights\label{sec:nonlin-drop}}

Many problems in mathematics, in physical sciences and in real-world
applications can be modeled as a \textit{convex feasibility problem }(CFP);
i.e., a problem of finding a point $x^{\ast}\in Q:=\cap_{i=1}^{m}Q_{i}$ in the
intersection of finitely many closed convex sets $Q_{i}\subseteq%
\mathbb{R}%
^{n}$ in the finite-dimensional Euclidean space. The literature on this
subject is enormously large, see, e.g., \cite{bb96} or \cite{bau-com-book} and
\cite{cegielski-book} and references therein.

Fully-simultaneous (parallel) algorithmic schemes for the CFP employ iterative
steps of the form
\begin{equation}
x^{k+1}=x^{k}+\lambda_{k}\left(  \sum_{i=1}^{m}w_{i}\left(  P_{Q_{i}}%
(x^{k})-x^{k}\right)  \right)  , \label{eq:simult-alg}%
\end{equation}
where $P_{\Omega}(x)$ stands for the \textit{orthogonal (nearest Euclidean
distance) projection} of a point $x$ onto the closed convex set $\Omega,$ the
parameters $\{w_{i}\}_{i=1}^{m}$ are a \textit{system of weights} such that
$w_{i}>0$ for all $i=1,2,\ldots,m$ and $\sum_{i=1}^{m}w_{i}=1$, and the
\textit{relaxation parameters }$\{\lambda_{k}\}_{k=0}^{\infty}$\textit{\ }are
user-chosen and, in most convergence analyses, must remain in a certain fixed
interval, in order to guarantee convergence.

For linear equations, represented by hyperplanes, i.e.,
\begin{equation}
Q_{i}=H_{i}:=\left\{  x\in%
\mathbb{R}%
^{n}\mid\langle a^{i},x\rangle=b_{i}\right\}  , \label{eq:hyper}%
\end{equation}
for $i=1,2,\ldots,m$, the orthogonal projection $P_{i}(z)$ of a point $z\in%
\mathbb{R}%
^{n}$ onto $H_{i}$ is
\begin{equation}
P_{i}(z)~=~z+\frac{b_{i}-\langle a^{i},z\rangle}{\left\Vert a^{i}\right\Vert
_{2}^{2}}a^{i}\,, \label{2.1.4}%
\end{equation}
where $a^{i}=(a_{j}^{i})_{j=1}^{n}\in%
\mathbb{R}%
^{n}$, $a^{i}\neq0$, and $b_{i}\in%
\mathbb{R}%
$ are the given data of the linear equations and $\Vert\cdot\Vert_{2}$ is the
Euclidean norm. The iterative steps of (\ref{eq:simult-alg}) then take the
form
\begin{equation}
x^{k+1}~=~x^{k}+\lambda_{k}\sum_{i=1}^{m}w_{i}\frac{b_{i}-\langle a^{i}%
,x^{k}\rangle}{\left\Vert a^{i}\right\Vert _{2}^{2}}a^{i}. \label{eq:cimi}%
\end{equation}
This algorithm was first proposed by Cimmino \cite{Cimmino} (read about the
profound impact of this paper on applied scientific computing in Benzi's paper
\cite{benzi}) and generalized to convex sets by Auslender \cite{AUSLENDER76}.

For the case of equal weights $w_{i}=1/m$ (\ref{eq:cimi}) becomes
\begin{equation}
x^{k+1}~=~x^{k}+\frac{\lambda_{k}}{m}\sum_{i=1}^{m}\frac{b_{i}-\langle
a^{i},x^{k}\rangle}{\left\Vert a^{i}\right\Vert _{2}^{2}}a^{i}. \label{2.1.6}%
\end{equation}

When the $m\times n$ system matrix $A=(a_{j}^{i})$ is sparse, as often happens
in some important real-world applications, only a relatively small number of
the elements $\{a_{j}^{1},a_{j}^{2},\ldots,a_{j}^{m}\}$ in the $j$-th column
of $A$ are nonzero, but in (\ref{2.1.6}) the sum of their contributions is
divided by the relatively large $m$ -- slowing down the progress of the
algorithm. This observation led us, in \cite{CEHN08}, to consider replacement
of the factor $1/m$ in (\ref{2.1.6}) by a factor that depends only on the
number of \textit{nonzero} elements in the set $\{a_{j}^{1},a_{j}^{2}%
,\ldots,a_{j}^{m}\}$. Specifically, for each $j=1,2,\ldots,n$, we denoted by
$s_{j}$ the number of nonzero elements in column $j$ of the matrix $A$,
assuming that all columns of $A$ are nonzero, thus $s_{j}\neq0,$ for all $j.$
Then we replaced (\ref{2.1.6}) by
\begin{equation}
x_{j}^{k+1}~=~x_{j}^{k}+\frac{\lambda_{k}}{s_{j}}\sum_{i=1}^{m}\frac
{b_{i}-\langle a^{i},x^{k}\rangle}{\left\Vert a^{i}\right\Vert _{2}^{2}}%
a_{j}^{i}\,,\text{ for  }j=1,2,\ldots,n.\label{2.1.62}%
\end{equation}
This iterative formula is the backbone of the proposed fully-simultaneous
Diagonally-Relaxed Orthogonal Projections (DROP) method for linear equations
in \cite{CEHN08}. Certainly, if the matrix $A$ is sparse, then the $s_{j}$
values will be much smaller than $m$ and using them instead of the large $m$
will enlarge the additive updates in the iterative process and lead to faster
initial convergence.

This leads naturally to the question whether the weights $w_{i}$ in
(\ref{eq:simult-alg}) be allowed to depend on the component index $j$ as
iterations proceed, without loosing the guaranteed convergence of the
algorithm? Or, phrased mathematically, may the iterations proceed according
to
\begin{equation}
x_{j}^{k+1}=x_{j}^{k}+\lambda_{k}\left(  \sum_{i=1}^{m}w_{ij}\left(  \left(
P_{Q_{i}}(x^{k})\right)  _{j}-x_{j}^{k}\right)  \right)  ,\text{ for
}j=1,2,\ldots,n, \label{eq:unsetteled}%
\end{equation}
where the parameters $\{w_{ij}\}_{i=1}^{m}$ form $n$ \textit{systems of
weights} such that, for $j=1,2,\ldots,n,$ $w_{ij}\geq0$ for all $i=1,2,\ldots
,m,$ and $\sum_{i=1}^{m}w_{ij}=1$? If such component-wise relaxation is
possible then we could use it to exploit sparsity of the underlying problem
and to control asynchronous (block) iterations.

The convergence of such a scheme, like (\ref{2.1.62}), for the linear case is
studied in \cite{CEHN08} but the general case of (\ref{eq:unsetteled})
remained unsettled.

With $T_{i}:=P_{Q_{i}},$ for all $i=1,2,\ldots,m,$ and with a fixed
$\lambda\in(0,2],$ such that $\lambda_{k}=\lambda,$ for all $k\geq0,$ Eq.
(\ref{eq:unsetteled}) describes the iterative process in
(\ref{eq:iter-process}) with the operator in (\ref{eq:nonlin-drop}). The
following theorem yields the weak convergence of sequences generated by the
iterative process (\ref{eq:unsetteled}).

\begin{theorem}
\label{thm:final}Let $w_{ij}\geq0$ with $\sum_{i=1}^{m}w_{ij}=1$, $i\in I$,
$j\in J$, and $\lambda\in(0,2)$. If $T_{i}:\mathcal{H}\rightarrow\mathcal{H}$
are CW-cutters and CW-WR operators having a common fixed point then any
sequence generated by the iterative process
\begin{equation}
x_{j}^{k+1}=x_{j}^{k}+\lambda\left(  \sum_{i=1}^{m}w_{ij}\left(  \left(
T_{i}(x^{k})\right)  _{j}-x_{j}^{k}\right)  \right)  ,\text{ for }%
j=1,2,\ldots,n\text{,} \label{e-final}%
\end{equation}
where $x^{0}\in\mathcal{H}$ is arbitrary, converges weakly to some point
$x^{\ast}\in\operatorname*{Fix}T=\bigcap_{j\in J}\operatorname*{Fix}{}^{j}T$.
\end{theorem}

\begin{proof}
Iteration (\ref{e-final}) can be written as $x_{j}^{k+1}=(T(x^{k}))_{j}$,
where $T$ is defined by (\ref{eq:nonlin-drop}) with $\lambda_{j}=\lambda$. By
Fact \ref{f9}, $T$ is $r$-CW-SQNE, where $r=\frac{{\textstyle2-\lambda}%
}{{\textstyle\lambda}}{\textstyle e}$ with $e:=(1,1,...,1)\in%
\mathbb{R}
^{n}$. By Fact \ref{f-CW-WR} and \cite[Corollary 5.3(i)]{CRZ18}, $T$ is CW-WR.
Thus, the theorem follows from Theorem \ref{f-WC}.
\end{proof}

\bigskip

It is, of course, possible to replace the constant relaxation parameter
$\lambda$ in (\ref{e-final}) with $\lambda_{k}\in\lbrack\varepsilon
,2-\varepsilon]$ for some positive $\varepsilon$ and the proof will remain
true. Therefore, since the operator $T$ in the above proof is a CW-WR
operator, and so are also orthogonal projections, Theorem \ref{thm:final} is a
generalization of the convergence theorem of the DROP method, \cite[Theorem
2.3]{CEHN08}.

\bigskip{}

\textbf{Acknowledgement}. We thank Eliahu Levy and Daniel Reem for some
fruitful discussions on this topic. The work of Yair Censor is supported by
the ISF-NSFC joint research program Grant No. 2874/19.

\hfill{}

\end{document}